\documentclass[a4paper,14pt]{article}[2000/05/19]
\usepackage[english]{babel}
\usepackage{amssymb,amsthm,amsmath}
\usepackage{inputenc}
\newtheorem{theorem}{Theorem}[section]

\newtheorem{proposition}[theorem]{Proposition}
\newtheorem{corollary}[theorem]{Corollary}
\theoremstyle{definition}

\theoremstyle{remark}
\newtheorem{remark}[theorem]{Remark}
\numberwithin{equation}{section}
\begin{document}
\large
\setcounter{page}{1}

\begin{center}
United Nations Educational, Scientific and Cultural Organization \\
and\\
International Atomic Energy Agency \\
THE ABDUS SALAM INTERNATIONAL CENTRE FOR THEORETICAL PHYSICS
\end{center}

\vspace{0.5 cm}

\begin{center}
\textbf{TOPOLOGIES ON CENTRAL EXTENSIONS OF VON NEUMANN ALGEBRAS}
\end{center}

\vspace{0.5 cm}

\begin{center}
Sh. A. Ayupov \footnote{Senior Associate of ICTP. Corresponding  author. sh\_ayupov@mail.ru}\\
\textit{Institute of Mathematics and Information  Technologies, Uzbekistan Academy of Sciences
Dormon yoli str., 29, 100125,  Tashkent,   Uzbekistan} \\
\textit{and} \\
\textit{The Abdus Salam International Centre for Theoretical Physics, Trieste, Italy}
\end{center}

\begin{center}
K. K. Kudaybergenov \footnote{karim2006@mail.ru}\\
\textit{Department of Mathematics, Karakalpak state university
Ch. Abdirov str.,1,  230113, Nukus,    Uzbekistan}
\end{center}

\begin{center}
and 
\end{center}

\begin{center}
R.~T.~Djumamuratov \footnote{rauazh@mail.ru}\\
\textit{Department of Mathematics, Karakalpak state university
Ch. Abdirov str.,1,  230113, Nukus,    Uzbekistan}
\end{center}

\vspace{0.5 cm}

\begin{abstract}
Given a von Neumann algebra $M$ we  consider the  central
exten\-sion $E(M)$ of $M.$  We introduce the topology
$t_c(M)$ on  $E(M)$ generated by a center-valued norm and prove that it coincides
with the topology of convergence locally in measure on $E(M)$ if and only if $M$ does not have
direct summands of type II. We also show that $t_c(M)$ restricted on the set $E(M)_h$ of self-adjoint elements of  $E(M)$ coincides with the order topology on $E(M)_h$  if and only if $M$ is a $\sigma$-finite type  I$_{fin}$
von Neumann algebra.
\end{abstract}

\vspace{0.5 cm}

\begin{center}
MIRAMARE --- TRIESTE
\end{center}

\newpage

\section{Introduction}

In the series of paper \cite{Alb1}-\cite{AK3}  we have considered
derivations on the algebra $LS(M)$ of locally measurable operators
affiliated with a von Neumann algebra $M,$ and on various
subalgebras of $LS(M).$ A complete description of derivations has
been obtained in the case of  von Neumann algebras of type I and
III.
A comprehensive survey of recent results concerning derivations on
various algebras of unbounded operators affiliated with von
Neumann algebras is presented in \cite{AK2}.
A general form  of automorphisms  on the algebra $LS(M)$  in the case of  von Neumann algebras of type I has been obtained in \cite{AK3}. In proof of the main results of the above  papers the crucial role is played by the co-called
central extensions  of von Neumann algebras and also by various topologies considered in \cite{AK1}.

 Let  $M$ be an arbitrary von Neumann algebra with the center $Z(M)$
 and let $LS(M)$ denote the algebra of all locally measurable operators with
 respect $M.$ We  consider the set   $E(M)$  of all elements  $x$ from  $LS(M)$ for which there exists a sequence of
mutually orthogonal central projections  $\{z_i\}_{i\in I}$ in  $M$ with $\bigvee\limits_{i\in I}z_i=\textbf{1},$
such that $z_i x\in M$ for all $i\in I.$
It is known  \cite{AK1} that  $E(M)$ is a *-subalgebra in  $LS(M)$ with the center
 $S(Z(M)),$ where   $S(Z(M))$ is  the algebra of all measurable operators
 with respect to   $Z(M),$ moreover,
  $LS(M)=E(M)$ if and only if $M$ does not have
direct summands of type II.

A similar notion (i.e. the algebra $E(\mathcal{A})$) for
arbitrary *-subalgebras $\mathcal{A}\subset LS(M)$  was independently
introduced  by M.A. Muratov and V.I. Chilin \cite{Mur1}.
The algebra  $E(M)$ is called
\textit{the central extension of} $M.$
It is known (\cite{AK1},
\cite{Mur1}) that an element
$x\in  LS(M)$ belongs to $E(M)$ if and only if there exists
 $f\in S(Z(M))$ such that    $|x|\leq f.$
Therefore for each
 $x\in E(M)$ one can define the following vector-valued norm
 $ ||x||=\inf\{f\in S(Z(M)): |x|\leq f\}.$
This center-valued norm naturally generates a topology on $E(M)$ which denoted by $t_c(M).$

In this paper we study the relationship between the topology $t_c(M)$ on $E(M)$  generated by the above center-valued
norm, the topology $t(M)$ --  of convergence locally in measure, and the order topology
$t_o(M)$ on $E(M)_h$.
We prove that $t_c(M)$ coincides
with the topology $t(M)$ on $E(M)$ if and only if $M$ does not have
direct summands of type II.
We show  that $t_c(M)$ coincides
with the order  topology on $E(M)_h$ if and only if $M$ is a $\sigma$-finite type  I$_{fin}$ algebra.

\section{Central extensions of von Neumann algebras}

Let  $H$ be a complex Hilbert space and let  $B(H)$ be the algebra
of all bounded linear operators on   $H.$ Consider a von Neumann
algebra $M$  in $B(H)$ with the operator norm $\|\cdot\|_M.$ Denote by
$P(M)$ the lattice of projections in $M.$

A linear subspace  $\mathcal{D}$ in  $H$ is said to be
\emph{affiliated} with  $M$ (denoted as  $\mathcal{D}\eta M$), if
$u(\mathcal{D})\subset \mathcal{D}$ for every unitary  $u$ from
the commutant
$$M'=\{y\in B(H):xy=yx, \,\forall x\in M\}$$ of the von Neumann algebra $M.$

A linear operator  $x$ on  $H$ with the domain  $\mathcal{D}(x)$
is said to be \emph{affiliated} with  $M$ (denoted as  $x\eta M$) if
$\mathcal{D}(x)\eta M$ and $u(x(\xi))=x(u(\xi))$
 for all  $\xi\in
\mathcal{D}(x).$

A linear subspace $\mathcal{D}$ in $H$ is said to be \emph{strongly
dense} in  $H$ with respect to the von Neumann algebra  $M,$ if

1) $\mathcal{D}\eta M;$

2) there exists a sequence of projections
$\{p_n\}_{n=1}^{\infty}$ in $P(M)$  such that
$p_n\uparrow\textbf{1},$ $p_n(H)\subset \mathcal{D}$ and
$p^{\perp}_n=\textbf{1}-p_n$ is finite in  $M$ for all
$n\in\mathbb{N},$ where $\textbf{1}$ is the identity in $M.$

A closed linear operator  $x$ acting in the Hilbert space $H$ is said to be
\emph{measurable} with respect to the von Neumann algebra  $M,$ if
 $x\eta M$ and $\mathcal{D}(x)$ is strongly dense in  $H.$ Denote by
 $S(M)$ the set of all measurable operators with respect to
 $M$ (see \cite{Seg}).

A closed linear operator $x$ in  $H$  is said to be \emph{locally
measurable} with respect to the von Neumann algebra $M,$ if $x\eta
M$ and there exists a sequence $\{z_n\}_{n=1}^{\infty}$ of central
projections in $M$ such that $z_n\uparrow\textbf{1}$ and $z_nx \in
S(M)$ for all $n\in\mathbb{N}$ (see \cite{Yea}).

It is well-known  \cite{Mur}, \cite{Yea} that the set $LS(M)$ of
all locally measurable operators with respect to $M$ is a unital
*-algebra when equipped with the algebraic operations of strong
addition and multiplication and taking the adjoint of an operator,
and contains $S(M)$ as a solid *-subalgebra.

Let $(\Omega,\Sigma,\mu)$  be a measure space and from now on
suppose
 that the measure $\mu$ has the  direct sum property, i. e. there is a family
 $\{\Omega_{i}\}_{i\in
J}\subset\Sigma,$ $0<\mu(\Omega_{i})<\infty,\,i\in J,$ such that
for any $A\in\Sigma,\,\mu(A)<\infty,$ there exist a countable
subset $J_{0 }\subset J$ and a set  $B$ with zero measure such
that $A=\bigcup\limits_{i\in J_{0}}(A\cap \Omega_{i})\cup B.$

 We denote by  $L^{0}(\Omega, \Sigma, \mu)$ the algebra of all
(equivalence classes of) complex measurable functions on $(\Omega,
\Sigma, \mu)$ equipped with the topology of convergence in
measure.

Consider the algebra  $S(Z(M))$  of operators which are measurable
with respect to the  center $Z(M)$ of the von Neumann algebra $M.$
Since  $Z(M)$ is an abelian von Neumann algebra  it
 is *-isomorphic to $ L^{\infty}(\Omega, \Sigma, \mu)$
   for an appropriate measure space $(\Omega, \Sigma, \mu)$. Therefore the algebra  $S(Z(M))$ coincides
   with $Z(LS(M))$ and  can be
 identified with the algebra $ L^{0}(\Omega, \Sigma, \mu)$ of all
 measurable functions on $(\Omega, \Sigma, \mu)$.

The basis of neighborhoods of zero in the topology of convergence locally in measure
    on $L^0(\Omega,\Sigma, \mu)$ consists of the sets
$$W(A,\varepsilon,\delta)=\{f\in L^0(\Omega,\Sigma, \mu):\exists B\in \Sigma, \, B\subseteq A, \,
\mu(A\setminus B)\leq \delta, $$
$$ f\cdot \chi_B \in L^{\infty}(\Omega,\Sigma, \mu),\,
||f\cdot \chi_B||_{L^{\infty}(\Omega,\Sigma, \mu)}\leq \varepsilon\},$$
where $\varepsilon, \delta>0, \, A\in \Sigma, \, \mu(A)<+\infty,$
and $\chi_B$ is the characteric
function of the set $B\in \Sigma.$

Recall the definition of the dimension functions on the lattice
$P(M)$ of projection from $M$ (see \cite{Mur}, \cite{Seg}).

By  $L_+$ we denote the set of all measurable  functions  $f:
(\Omega,\Sigma, \mu)\rightarrow [0,{\infty}]$ (modulo functions
equal to zero $\mu$-almost everywhere ).

 Let  $M$ be an arbitrary von Neumann algebra with the center $Z(M)\equiv L^\infty(\Omega,\Sigma, \mu).$
Then there exists a map  $D:P(M)\rightarrow
L_{+}$ with the following properties:

(i) $d(e)$ is a finite function if only if the projection  $e$ is finite;

(ii) $d(e+q)=d(e)+d(q)$  for $p, q \in P(M),$ $eq=0;$

(iii) $d(uu^*)=d(u^*u)$ for every  partial isometry  $u\in M;$

(iv) $d(ze)=zd(e)$ for all $z\in P(Z(M)), \,\, e\in P(M);$

(v) if  $\{e_{\alpha}\}_{\alpha \in J}, \,\,\, e\in P(M) $ and $e_{\alpha}\uparrow e,$ then
$$d(e)=\sup \limits_{\alpha \in J}d(e_{\alpha}).$$
This map  $d:P(M)\rightarrow L_+,$ is a called the \emph{dimension functions} on  $P(M).$

Recall that for an element   $x\in LS(M)$ the projection defined as
$$
c(x)=\inf\{z\in P(Z(M)): zx=x\}
$$
 is called  \emph{the central
cover of }$x.$

\begin{remark} \label{R}

Let $M$ be a type I von Neumann algebra. If  $p, q\in P(M)$
abelian projections are faithful (i.e. with $c(p)=c(q)=\textbf{1},$) then the property
(iii) implies  that $0<d(p)(\omega)=d(q)(\omega)<\infty$ for
$\mu$-almost every  $\omega\in\Omega.$ Therefore replacing $d$ by
$d(p)^{-1}d$ we can assume that  $d(p)=c(p)$ for every faithful abelian
projection  $p\in P(M).$ Thus for all  $e\in P(M)$ we have that
$d(e)\geq c(e).$
\end{remark}

The basis of neighborhoods of zero in \emph{the topology $t(M)$ of  convergence
locally in measure}   on $LS(M)$ consists (in the above notations)
of the following sets
$$V(A,\varepsilon,\delta)=\{x\in LS(M):\exists p\in P(M), \, \exists z\in P(Z(M)),  \,
xp \in M, $$
$$ ||xp||_{M}\leq \varepsilon, \,\, z^{\bot}\in W(A,\varepsilon,\delta), \,\, d(zp^{\bot})\leq \varepsilon z\},$$
where
 $\varepsilon, \delta>0, \, A\in \Sigma, \, \mu(A)<+\infty.$

The topology  $t(M)$ is  metrizable if and only if the center  $Z(M)$
is   $\sigma$-finite (see \cite{Mur}).

Given an arbitrary  family  $\{z_i\}_{i\in I}$ of mutually orthogonal
central projections in $M$ with $\bigvee\limits_{i\in
I}z_i=\textbf{1}$ and a  family of elements $\{x_i\}_{i\in I}$ in
$LS(M)$ there exists a unique element $x\in LS(M)$ such that $z_i
x=z_i x_i$ for all $i\in I.$ This element is denoted by
$x=\sum\limits_{i\in I}z_i x_i.$

We  denote by  $E(M)$  the set of all elements  $x$ from  $LS(M)$ for which there exists a sequence of
mutually orthogonal central projections  $\{z_i\}_{i\in I}$ in  $M$ with $\bigvee\limits_{i\in I}z_i=\textbf{1},$
such that $z_i x\in M$ for all $i\in I,$ i.e.
 $$E(M)=\{x\in LS(M): \exists z_i\in P(Z(M)), z_iz_j=0, i\neq j, \bigvee\limits_{i\in I}z_i=\textbf{1},
 z_i x\in M, i\in I\},$$
where $Z(M)$ is the center of $M.$

It is known  \cite{AK1} that  $E(M)$ is  *-subalgebras in  $LS(M)$ with the center
 $S(Z(M)),$ where   $S(Z(M))$ is  the algebra of all measurable operators
 with respect to   $Z(M),$ moreover,
  $LS(M)=E(M)$ if and only if $M$ does not have
direct summands of type II.

A similar notion (i.e. the algebra $E(\mathcal{A})$) for
arbitrary *-subalgebras $\mathcal{A}\subset LS(M)$  was independently
introduced recently by M.A. Muratov and V.I. Chilin \cite{Mur1}.
The algebra  $E(M)$ is called
\textit{the central extension of} $M.$

It is known (\cite{AK1},
\cite{Mur1}) that an element
$x\in  LS(M)$ belongs to $E(M)$ if and only if there exists
 $f\in S(Z(M))$ such that    $|x|\leq f.$
Therefore for each
 $x\in E(M)$ one can define the following vector-valued norm
 \begin{equation}
 \label{norm}
 ||x||=\inf\{f\in S(Z(M)): |x|\leq f\}
\end{equation}
and this norm satisfies the following conditions:

$1) \|x\|\geq 0; \|x\|=0\Longleftrightarrow x=0;$

$2)  \|f x\|=|f|\|x\|;$

$ 3)  \|x+y\|\leq\|x\|+\|y\|;$

$4) ||x y||\leq ||x||||y||;$

$5) ||xx^{\ast}||=||x||^2$
\newline
    for all $x,y\in E(M), f\in S(Z(M)).$

\section{Topologies on the central extensions of von Neumann algebras}

 Let  $M$ be an arbitrary von Neumann algebra with the center $Z(M)\equiv L^\infty(\Omega,\Sigma, \mu).$
On the space  $E(M)$ we consider the following sets:
\begin{equation}
 \label{base1}
O(A, \varepsilon, \delta)=\left\{x \in E(M): ||x||\in W(A, \varepsilon, \delta)\right\},
\end{equation}
where  $\varepsilon ,\delta > 0,\,\,\,A \in \sum ,\,\,\,\mu \left( A \right) <
+ \infty $.

The following proposition gives elementary properties of the sets
$O(A, \varepsilon, \delta),$ which immediately follow from the corresponding
properties of the sets  $V(A, \varepsilon, \delta)$ (see \cite[Proposition 3.5.1]{Mur}).

\begin{proposition}\label{Ele} Let  $\varepsilon, \varepsilon_j>0$ and  $\delta,
\delta_j>0, j=1, 2, A\in \Sigma, \mu(A)<\infty.$ Then

i)  $\lambda O(A, \varepsilon, \delta)=O(A, |\lambda|\varepsilon, \delta)$ для всех $\lambda\in \mathbb{C},
\lambda\neq 0;$

ii)   $ O(A, \varepsilon_1, \delta_1)\subseteq O(A, \varepsilon_2, \delta_2)$ если
 $\varepsilon_1\leq\varepsilon_2, \delta_1\leq\delta_2;$

iii)   $O(A, \varepsilon_1, \delta_1)+O(A, \varepsilon_2, \delta_2)\subseteq
O(A, \varepsilon_1+\varepsilon_2, \delta_1+\delta_2);$

iv)   $O(A, \varepsilon_1, \delta_1)O(A, \varepsilon_2, \delta_2)\subseteq
O(A, \varepsilon_1\varepsilon_2, \delta_1+\delta_2);$

v)   $O^\ast(A, \varepsilon, \delta)=O(A, \varepsilon, \delta),$ where
 $O^\ast(A, \varepsilon, \delta)=\{x^{\ast}: x\in O(A, \varepsilon, \delta)\};$

vi)   $\bigcap\{O(A, \varepsilon, \delta): \varepsilon>0, \delta>0,
A\in \Sigma, \mu(A)<\infty\}=\{0\}.$
\end{proposition}

From Proposition \ref{Ele} it follows that the system  of sets
\begin{equation}
 \label{base2}
\{x+O(A, \varepsilon, \delta)\},
\end{equation}
where $x\in  E(M), \varepsilon>0, \delta>0,
A\in \Sigma, \mu(A)<\infty$, defines on  $E(M),$ a Hausdorff vector topology $t_c(M),$
for which the sets (\ref{base2}) form the base of neighborhoods of the element $x\in E(M).$
Moreover  in this topology the   involution is continuous and the multiplication  is jointly continuous, i.e. $(E(M), t_c(M))$ is a topological *-algebra.
From \cite[Proposition 5.3]{AK2} it follows that $(E(M), t_c(M))$ is  complete.

Thus we obtain the following result.

\begin{proposition}\label{A1}
i)   $(E(M), t_c(M))$ is a complete topological *-algebra;

ii) $M$ is a  $t_c(M)$-dense in  $E(M).$
\end{proposition}

\begin{proof}  i) is proved above.

ii).  Let  $x\in E(M)$ and $A\in \Sigma, \, \mu(A)<\infty,$
$\varepsilon, \delta>0.$ For $n\in \mathbb{N}$ put
\[
B_n=\{\omega\in A: ||x||(\omega)\leq n\}.
\]
Since
 $\mu(A\setminus B_n)\rightarrow 0$
as $n\rightarrow\infty$ there is  $k\in \mathbb{N}$ such that
  $\mu(A\setminus B_k)\leq\varepsilon.$
Put  $x_k=\chi_{B_k}x.$ Then
  \[
  ||x_k||\leq k\textbf{1}
  \]
  and
  \[
||x-x_k||\chi_{B_k}=||x\chi_{B_k}-x_k\chi_{B_k}||=||x\chi_{B_k}-x\chi_{B_k}||=0.
\]
Thus  $x_k\in x+O(A, \varepsilon, \delta).$ This means that
$\overline{M}^{t_c(M)}=E(M).$
The proof is complete.
\end{proof}

\begin{remark}\label{R2}
Note that if  $M$ is a commutative von Neumann algebra then  $||x||=|x|$ for each $x\in E(M),$
and therefore  $O(A, \varepsilon, \delta)=W(A, \varepsilon, \delta)$
for all  $\varepsilon ,\delta > 0,\,\,\,A \in \sum ,\,\,\,\mu \left( A \right) <
+ \infty.$
Hence the topology $t_c(M)$ on $E(M)$  coincides  with
the topology of convergence locally  in measure  $t(M).$

If $M$ is a factor,
then $E(M)=M$ and $t_c(M)=t_{\|\cdot\|_M},$ where  $t_{\|\cdot\|_M}$ uniform topology on $M.$ \end{remark}

\begin{proposition}\label{Con}
i) A net  $\{p_\alpha\}\subset P(M)$
converges to zero with respect to the topology  $t_c(M)$ if and only if
$c(p_\alpha)\stackrel{t(Z(M))}\longrightarrow 0,$
where $t(Z(M))$ is the topology of convergence locally in measure  on $Z(M).$

ii) A net   $\{x_\alpha\}\subset E(M)$
converges to zero with respect to the topology $t_c(M)$ if and only if
$e_\lambda^\perp(|x_\alpha|)\stackrel{t_c(M)}\longrightarrow 0$ for any $\lambda>0,$
where  $\{e_\lambda(|x_\alpha|)\}$ is a spectral projections family for the operator  $x_\alpha.$
\end{proposition}

\begin{proof} i) The proof immediately follows from the definition of the topology
$t_c(M)$ and the equality  $||p||=c(p),\, p\in P(M).$

ii) Let  $x_\alpha\stackrel{t_c(M)}\longrightarrow 0$ and $\lambda>0.$ Take any
$A\in \Sigma, \mu(A)<\infty, 0<\varepsilon<\lambda/2, \delta>0.$
Since  $||x_\alpha||\stackrel{t(M)}\longrightarrow 0$, then there exists
 $\alpha_0$ such that $||x_\alpha||\in W(A, \varepsilon, \delta)$ for each  $\alpha\geq\alpha_0.$
Therefore  there exists $B_\alpha\in \Sigma, B_\alpha\subseteq A$
such that  $\mu(A\setminus B_\alpha)\leq\delta,$
$||||x_\alpha||\chi_{B_{\alpha}}||_M\leq\varepsilon.$ Thus
$|||x_\alpha|\chi_{B_{\alpha}}||_M\leq\varepsilon,$ i.e.
$|x_\alpha|\chi_{B_{\alpha}}\leq\varepsilon\chi_{B_{\alpha}}.$
Since  $\varepsilon<\frac{\textstyle \lambda}{\textstyle 2}$ then from the last inequality
we have that
 $c(e_\lambda^\perp(|x_\alpha|))\chi_{B_{\alpha}}=0.$  The inequality
 $\mu(A\setminus B_\alpha)\leq\delta$ implies that
  $c(e_\lambda^\perp(|x_\alpha|))\in W(A, \varepsilon, \delta),$
i.e. $c(e_\lambda^\perp(|x_\alpha|))\stackrel{t(Z(M))}\longrightarrow 0.$ Thus
 $e_\lambda^\perp(|x_\alpha|)\stackrel{t_c(M)}\longrightarrow 0.$

Now let   $e_\varepsilon^\perp(|x_\alpha|)\stackrel{t_c(M)}\longrightarrow 0$ and  $0<\varepsilon<1, \delta>0.$
Then  $c(e_\varepsilon^\perp(|x_\alpha|))\stackrel{t(M)}\longrightarrow 0.$
Therefore there exists  $\alpha_0$ such that
 $c(e_\varepsilon^\perp(|x_\alpha|))\in W(A, \varepsilon, \delta)$ for all
 $\alpha\geq\alpha_0.$
Hence there exists $B_\alpha\in \Sigma, B_\alpha\subseteq A$
such that  $\mu(A\setminus B_\alpha)\leq\delta,$
$||c(e_\varepsilon^\perp(|x_\alpha|))\chi_{B_{\alpha}}||_M\leq\varepsilon<1.$ Thus
$c(e_\varepsilon^\perp(|x_\alpha|))\chi_{B_{\alpha}}=0,$ i.e.
$|x_\alpha|\chi_{B_{\alpha}}\leq\varepsilon\chi_{B_{\alpha}}.$
Therefore
\[
||||x_\alpha||\chi_{B_{\alpha}}||_M\leq\varepsilon
\]
and
\[
\mu(A\setminus B_\alpha)\leq\delta.
\]
Thus
  $||x_\alpha||\in W(A, \varepsilon, \delta),$
  i.e.
   $||x_\alpha||\stackrel{t(Z(M))}\longrightarrow 0.$ Therefore
 $x_\alpha\stackrel{t_c(M)}\longrightarrow 0.$
The proof is complete.
\end{proof}

Let  $t(M)$ denote  the topology on $E(M)$ induced by the topology $t(M)$ from $LS(M).$

\begin{proposition}\label{Com}
The topology $t_c(M)$ is stronger than
the topology $t(M)$ of convergence locally in measure.
\end{proposition}

\begin{proof}   It is sufficient to show that
\begin{equation}
 \label{baza}
O(A, \varepsilon, \delta) \subset V(A, \varepsilon, \delta).
\end{equation}

Let  $x \in O(A, \varepsilon, \delta),$
i.e. $||x|| \in W(A, \varepsilon, \delta).$
Then there exists  $B \in \Sigma$ such that
 \[
 B \subseteq A,\,\,\,\mu(A\setminus
  B) \le \delta,
  \]
 and
\[
||x||\chi_B \in L^\infty(\Omega, \Sigma, \mu),\,\,||\|x\|\chi_B||_{M}\le \varepsilon.
\]

Put  $z =p=\chi_B.$
Then  $||xp||=||x\chi_B||=||x||\chi_B\in L^\infty(\Omega, \Sigma, \mu),$ i.e. $xp
\in M$ and moreover $||xp||_M \le \varepsilon.$ Since  $\mu(A\setminus B) \le \delta$ and
$z^\perp \chi_B = \chi _B^\perp \chi_B =
0,$ one has $z^\perp\in W(A, \varepsilon, \delta).$ Therefore
\[
||xp||_M\leq\varepsilon,\,\,z^\perp\in W(A, \varepsilon, \delta),\,\, zp^\perp=\chi _B \chi_B^\perp=0
\]
and hence
$
x \in V(A, \varepsilon, \delta),
$
i.e. $O(A, \varepsilon, \delta) \subset V(A, \varepsilon, \delta).$ The proof
is complete. \end{proof}

\begin{proposition}\label{BASA}
If $M$ is a type  I or III  von Neumann algebra and
$0<\varepsilon<1,$ then
\[
O(A, \varepsilon, \delta)=V(A, \varepsilon, \delta).
\]
\end{proposition}

\begin{proof}  From above  \eqref{baza} we have that
$O(A, \varepsilon, \delta) \subset V(A, \varepsilon, \delta).$
Therefore it is sufficient to show that  $V(A, \varepsilon, \delta)
\subset O (A, \varepsilon, \delta).$

Let  $x \in V(A, \varepsilon, \delta).$ Then there exist  $p
\in P(M)$ and $z \in P(Z(M))$
such that
\[
xp \in M,\,\,\,\,||xp||_M \le \varepsilon
,\,\,\,\,z^\perp \in W(A, \varepsilon, \delta),\,\,\,d(zp^\perp)\le \varepsilon z.
\]

If    $M$ is of type  I then Remark \ref{R}  implies that  $d(zp^\perp)\geq c(zp^\perp).$
Now from   $d(zp^\perp) \le \varepsilon z$ it follows that
 $c(zp^\perp)\leq\varepsilon z.$
From  $0<\varepsilon<1$ we obtain that    $zp^\perp=0.$

If  $M$ is of type  III then the finiteness of the projection
  $zp^\perp$ implies that
 $zp^\perp=0.$

 Thus
 $z = zp.$
Put $z=\chi_E$ for an appropriate $E\in \Sigma.$ Since
 $z^\perp \in W(A, \varepsilon, \delta)$
one has that  $\chi_{\Omega \setminus E} \in W(A, \varepsilon, \delta).$
  Thus there exists  $B \in \Sigma $ such that  $B \subseteq
A,\,\,\,\,\mu(A\setminus B) \le \delta,$ $|\chi_{\Omega \setminus E}\chi _B| \le \varepsilon < 1.$
 Hence   $\chi_B\leq \chi_E.$ So we obtain
\[
||x||\chi_B \le ||x||\chi_E=||x||z=||xz||=||xzp||=||xp|| \le \varepsilon.
\]
This means that  $x \in O(A, \varepsilon, \delta).$
The proof is complete. \end{proof}

Proposition \ref{BASA} implies that following

\begin{theorem}\label{IorIII}
If  $M$ is a type I or III von Neumann algebra then the topologies $t(M)$ and $t_c(M)$
coincide.
\end{theorem}

\begin{proposition}\label{II} If  $M$ is of type II then  $t(M)<t_c(M).$
\end{proposition}

\begin{proof} Since $M$ is a type II then there exists a decreasing sequence  of
projections $\{p_n\}$ in  $M$
such that $c(p_n)=\textbf{1}$ and  $d(p_n)=\frac{\textstyle 1}{\textstyle 2^n}$
for all $n\in \mathbb{N}.$ Then  $\{p_n\}$ converges to zero with respect to the topology locally in measure. Indeed
 take any  neighborhood  of zero $V(A, \varepsilon, \delta)$ in the topology $t(M).$
Put $z=\textbf{1}, p=p_k^\perp,$
where the number $k$ is such that  $\frac{\textstyle 1}{\textstyle 2^k}<\varepsilon.$
For
$n\geq k$
we have that
\[
p_np=p_n p_k^\perp=(p_n p_k)p_k^\perp=0,
\]
\[
 z^\perp\in W(A,\varepsilon, \delta)
 \]
 and
 \[
 d(zp^\perp)=d(p_k)=\frac{\textstyle 1}{\textstyle 2^k}\textbf{1}\leq\varepsilon z.
\]
This means that  $p_n\in V(A,\varepsilon, \delta)$ for all $n\geq k,$ i.e.
$\{p_n\}$ converges to zero with respect to the topology locally in measure.

On the other hand  the equality $c(p_n)=\textbf{1}$
implies that $||p_n||=\textbf{1}.$ Thus a sequence  $\{p_n\}$
does not converges to zero in the topology  $t_c(M).$ Hence
$t(M)<t_c(M).$ The proof is complete.
\end{proof}

Theorem \ref{IorIII} and Proposition \ref{II} imply the following
result
which describes the class of von Neumann algebras $M$ for which the topologies $t(M)$ and $t_c(M)$
coincide.

\begin{theorem}\label{A2}
The following conditions on a given von Neumann algebra $M$ are equivalent:

i)  $t(M)=t_c(M);$

ii)  $M$ does not have  direct summands of type   II.
\end{theorem}

By  $E(M)_h$ we denote the set of all selfadjoint elements in  $E(M).$
A net  $\{x_\alpha\}_{\alpha\in I}\subset E(M)_h$ is called $(o)$-convergent to  $x\in E(M)_h$ (denoted
$x_\alpha\stackrel{(o)}\longrightarrow x$), if there exist nets
$\{a_\alpha\}_{\alpha\in I}$ and  $\{b_\alpha\}_{\alpha\in I}$ in  $E(M)_h,$
 such that  $a_\alpha\leq x_\alpha\leq b_\alpha$ for each  $\alpha\in I$ and
 $a_\alpha\uparrow x,$ $b_\alpha\downarrow x.$
The strongest topology on $E(M)_h$ for which
 $(o)$-convergence of nets implies their convergence in the topology is called
 \emph{the order topology}, or
 \emph{the $(o)$-topology}, and is denoted by $t_o(M).$

Let  $t_{ch}(M)$ (respectively $t_h(M)$)
denote the topology on $E(M)_h$ induced by the topology $t_c(M)$ (respectively $t(M)$) from $E(M).$

We now  describe class of von Neumann algebras $M$ for which the topologies $t_{c}(M)$ and $t_o(M)$
coincide.

\begin{theorem}\label{A3}
(i) $t_{ch}(M)\leq t_o(M)$
if and only if  $M$ is of type I$_{fin};$

(ii) $t_{ch}(M)= t_o(M)$
if and only if $M$ is a $\sigma$-finite type  I$_{fin}$ algebra.
\end{theorem}

\begin{proof} (i) Let   $t_{ch}(M)\leq t_o(M).$
If the algebra  $M$ does not has type I$_{fin}$ then there exists a nonzero projection
$z\in P(Z(M))$
and a sequence of mutually orthogonal projections
 $\{p_n\}_{n=1}^{\infty}$ in  $M$ with  $c(p_n)=z, n\in\mathbb{N}.$
Then
$p_n\stackrel{(o)}\longrightarrow 0,$
 and therefore  $p_n\stackrel{t_{ch}(M)}\longrightarrow 0.$ Hence
 $||p_n||\stackrel{t(Z(M))}\longrightarrow 0.$ Since
 $||p_n||=c(p_n)=z$ it follows that    $z=0,$ this is a contradiction with $z\neq 0.$
Hence  $M$ is a type I$_{fin}$ algebra.

Conversely let  $M$ be a type  I$_{fin}$ algebra. Then by \cite[Proposition  1.1]{AK1}
we have that  $LS(M)=E(M).$ Thus   theorem  \ref{A2} implies that
$t_{ch}(M)=t_h(M).$  Since
$t_h(M)\leq t_o(M)$ (see \cite[Theorem 1 (i)]{MurUJM}) then
$t_{ch}(M)\leq t_o(M).$

(ii) If $t_{ch}(M)= t_o(M)$ then $M$ is a type  I$_{fin}$ algebra (see  (i)).
Again using the theorem \ref{A2} we have that
 $t_{ch}(M)= t_h(M).$ Thus
  $t_{h}(M)= t_o(M).$ Now by  \cite[Theorem 1 (ii)]{MurUJM} follows that  $M$ is a $\sigma$-finite algebra.

Conversely let   $M$ be a  $\sigma$-finite type  I$_{fin}$ algebra. Then
by theorem \ref{A2} we have that
 $t_{ch}(M)= t_h(M)$ and  by  \cite[Theorem 1 (ii)]{MurUJM} we obtain that
  $t_{h}(M)= t_o(M).$ Hence
$$
t_{ch}(M)= t_h(M)= t_o(M).
$$
The proof is complete.
\end{proof}

Theorem \ref{A3} yields the  following corollary.

\begin{corollary}\label{Co} The following assertions are true:

(i) If  $M$ is  a  $\sigma$-finite von Neumann algebra but is not  type
I$_{fin},$ then  $t_{o}(M)< t_{h}(M)$;

(ii) If  $M$ is not a $\sigma$-finite von Neumann algebra but is type
I$_{fin},$ then  $t_{h}(M)< t_{o}(M).$
\end{corollary}

\begin{proposition}\label{LcM}
The topology $t_c(M)$ is locally convex if and only if $M$ is *-isomorphic to the
$C^\ast$-product $\bigoplus\limits_{j\in J}M_j,$ where $M_j$ are factors.
\end{proposition}

\begin{proof}
Let  $t_c(M)$ be a locally convex topology on $E(M).$ Since $t_c(M)$
induces the topology $t(Z(M))$ on $Z(E(M))=S(Z(M)),$
we have that $(S(Z(M)), t(Z(M)))$ is a locally convex space. It follows from
\cite[12, Ch. V, \S 3]{SXAC} that $Z(M)$ is an atomic von Neumann algebra. Hence, the algebra
 $M$ is *-isomorphic to the
$C^\ast$-product $\bigoplus\limits_{j\in J}M_j,$ where $M_j$ are factors for all $j\in J.$

Conversely, let $M=\bigoplus\limits_{j\in J}M_j,$ where $M_j$ are factors. Then
$$
E(M_j)=M_j, t_c(M)=t_{\|\cdot\|_{M_{j}}}, E(M)=\prod\limits_{j\in J}M_j
$$
and, hence the topology $t_c(M)$ is a Tychonoff product of the normed topologies
$t_{\|\cdot\|_{M_j}},$ that is, $t_c(M)$ is a locally convex topology.
The proof is complete. \end{proof}

Similarly, we obtain the following

\begin{proposition}\label{M}
The topology $t_c(M)$ can be normed if and only if $M=\bigoplus\limits_{j=1}^{n}M_j,$ where $M_j$ are factors,
$j=\overline{1, n}, n\in \mathbb{N}.$
\end{proposition}

\section*{Acknowledgments}
Part of this work was done within
the framework of the Associateship Scheme of the Abdus Salam
International Centre  for Theoretical Physics (ICTP), Trieste,
Italy. The first author thank ICTP for providing financial support
 and all facilities (July-August, 2011). This work is supported in part by
the DFG AL 214/36-1 project (Germany).

\end{document}